\documentclass[12pt]{amsart}
\usepackage{amssymb,graphicx,amscd,accents,array,url}
\usepackage[usenames,dvipsnames]{color}
\calclayout
\allowdisplaybreaks

\newcommand{\PSig}{P_{\Sigma_h}}
\newcommand{\PSigo}{P_{\0\Sigma_h}}
\newcommand{\PV}{P_{\0V_h}}
\newcommand{\PS}{P_{S_h}}
\newcommand{\PiV}{\varPi_h^V}

\newcommand{\ignore}[1]{}

\numberwithin{figure}{section}
\numberwithin{table}{section}
\newcommand\pd[2]{\frac{\partial#1}{\partial #2}}

\newcommand\grad{\operatorname{grad}}
\renewcommand\div{\operatorname{div}}
\newcommand\curl{\operatorname{curl}}
\newcommand\rot{\operatorname{rot}}

\newcommand\Lap{\operatorname{\Delta}}

\newcommand\gradh{\operatorname{\grad}_h}
\newcommand\gradoh{\operatorname{\grad}^\circ_h}

\newcommand\R{\mathbb{R}}

\newcommand\x{\times}

\renewcommand\P{{\mathcal P}}
\newcommand\Th{{\mathcal T}_h}

\newcommand\Hdiv{H(\div)}
\newcommand\Hrot{H(\rot)}
\newcommand\Hodiv{\0H(\div)}
\newcommand\Horot{\0H(\rot)}

\newcommand{\<}{\langle}
\renewcommand{\>}{\rangle}
\newcommand{\0}{\ring}

\newcommand\uu{\boldsymbol u}
\newcommand\vv{\boldsymbol v}
\newcommand\VV{V}
\newcommand\ww{\boldsymbol w}

\newcommand\ff{\boldsymbol f}
\newcommand\nn{\boldsymbol n}
\newcommand\ssb{\boldsymbol s}

\definecolor{DarkRed}{rgb}{0.45,0.08,0.0}
\definecolor{DarkBlue}{rgb}{0,0.08,0.45}

\numberwithin{equation}{section}

\newtheorem{thm}{Theorem}[section]

\newtheorem{lem}[thm]{Lemma}
\newtheorem{cor}[thm]{Corollary}

\theoremstyle{remark}
\newtheorem*{remark}{Remark}

\begin{document}
\title[Mixed Finite Element Approximation of the Vector Laplacian]
{Mixed finite element approximation of the vector Laplacian with Dirichlet boundary conditions}
%    author one information
%\author[D. N. Arnold, R. S. Falk, and R. Winther]
\author{Douglas N. Arnold}
\address{Institute for Mathematics and its Applications
and School of Mathematics,
University of Minnesota, Minneapolis, MN 55455}
\email{arnold@ima.umn.edu}
\urladdr{http://www.ima.umn.edu/\char'176arnold/}
\thanks{}
%    author two information
\author{Richard S. Falk}
\address{Department of Mathematics,
Rutgers University, Piscataway, NJ 08854}
\email{falk@math.rutgers.edu}
\urladdr{http://www.math.rutgers.edu/\char'176falk/}
\thanks{The work of the first author was supported in part by NSF grant 
DMS-0713568.
The work of the second author was supported in part by NSF grant DMS-0910540.
The work of the third author was supported in part by NSF grant DMS-1014817.
This work was primarily carried out at the Institute for Mathematics and its
Applications at the University of Minnesota where the authors were visiting.
}
%    author three information
\author{Jay Gopalakrishnan}
\address{Portland State University, PO Box 751, Portland, 
OR 97207-0751}
\email{gjay@pdx.edu}
\urladdr{http://web.pdx.edu/~gjay/}

%\address{Centre of Mathematics for Applications
%and Department of Informatics,
%University of Oslo, 0316 Oslo, Norway}
%\email{ragnar.winther@cma.uio.no}
%\urladdr{http://heim.ifi.uio.no/\char'176rwinther/}
\thanks{}
%    \subjclass is required.
\thanks{}
\begin{abstract}
  We consider the finite element solution of the vector Laplace equation on a
  domain in two dimensions.  For various choices of boundary conditions, it is
  known that a mixed finite element method, in which the rotation of the
  solution is introduced as a second unknown, is advantageous, and appropriate
  choices of mixed finite element spaces lead to a stable, optimally
  convergent discretization.  However, the theory that leads to these
  conclusions does not apply to the case of Dirichlet boundary conditions, in
  which both components of the solution vanish on the boundary.  We show, by
  computational example, that indeed such mixed finite elements do not perform
  optimally in this case, and we analyze the suboptimal convergence that does
  occur.  As we indicate, these results have implications for the solution of
  the biharmonic equation and of the Stokes equations using a mixed
  formulation involving the vorticity.
\end{abstract}

\maketitle

\section{Introduction}\label{sec:intro}

We consider the vector Laplace equation (Hodge Laplace equation for 1-forms)
on a two-dimensional domain $\Omega$.  That is, given a vector field $\ff$ on
$\Omega$, we seek a vector field $\uu$ such that
\begin{equation}
\label{vec-Lap}
\curl \rot \uu - \grad \div \uu = \ff \quad \text{in } \Omega.
\end{equation}
(Notations are detailed at the end of this introduction.)
A weak formulation of a boundary value problem for this equation
seeks the solution $\uu$ in a subspace $H\subset \Hrot\cap \Hdiv$
satisfying
\begin{equation}
 \label{weak-form}
(\rot\uu ,\rot \vv) + (\div\uu,\div\vv) = (\ff,\vv), \quad \vv\in H.
\end{equation}
If $H$ is taken to be $\Horot\cap \Hdiv$, the variational formulation implies
the equation \eqref{vec-Lap} together with the electric boundary conditions
\begin{equation}\label{ebc}
\uu\cdot\ssb =0, \quad \div\uu=0 \quad\text{on
 $\partial\Omega$}.
\end{equation}
Magnetic boundary conditions, $\uu\cdot\nn =0$, $\rot\uu=0$, result if instead
the subspace $H$ in the weak formulation is taken to be $\Hrot\cap \Hodiv$.
(The terms electric and magnetic derive from the close relation of the Hodge
Laplacian and Maxwell's equations.)  If the domain $\Omega$ is
simply-connected, both these boundary value problems are well-posed.
(Otherwise, $H$ contains a finite dimensional subspace consisting of vector
fields which satisfy the boundary conditions and have vanishing rotation and
divergence with dimension equal to the number of holes in the domain, and each
problem can be rendered well-posed by replacing $H$ with the orthogonal
complement of this space.)

Even when the domain is simply connected, finite element methods based on
\eqref{weak-form} are problematic.  For example, on a non-convex polygon, approximations
using continuous piecewise linear functions converge to a function different
from the solution of the boundary value.  See
\cite[\S~2.3.2]{bulletin} for more details.  A convergent finite
element method can be obtained by discretizing a \emph{mixed} formulation with
a \emph{stable} choice of elements.  The mixed formulation for the electric
boundary value problem seeks $\sigma\in H^1$, $\uu\in \Hdiv$ such that
\begin{gather*}
(\sigma, \tau) - (\uu, \curl \tau) =0, \quad \tau \in H^1,
\\
(\curl \sigma, \vv) + (\div \uu, \div \vv) = (\ff,\vv), \quad
\vv \in \Hdiv.
\end{gather*}
On a simply connected domain, this problem has a unique solution for any
$L^2$ vector field $\ff$; $\uu$ solves \eqref{vec-Lap} and \eqref{ebc} and
$\sigma=\rot u$.  To discretize, we choose finite element
spaces $\Sigma_h\subset H^1$, $\VV_h \subset H(\div)$, 
indexed by a sequence of positive numbers $h$ tending
to $0$,
and determine
 $\sigma_h \in \Sigma_h$, 
$\uu_h \in {\VV}_h$ by
\begin{gather}\label{mm1}
(\sigma_h, \tau) - (\uu_h, \curl \tau) =0, \quad \tau \in \Sigma_h,
\\\label{mm2}
(\curl \sigma_h, \vv) + (\div \uu_h, \div \vv) = (\ff,\vv), \quad
\vv \in {\VV}_h.
\end{gather}
In order to obtain a stable numerical method, the finite element spaces
$\Sigma_h$ and $\VV_h$ must be chosen appropriately.  A stable method is
obtained by choosing $\Sigma_h$ to be the Lagrange elements of any degree
$r\ge 1$ and $\VV_h$ to be the Raviart--Thomas elements of the same degree $r$
(where the case $r=1$ refers to the lowest order Raviart--Thomas elements).
In the notation of \cite{bulletin}, $\Sigma_h\x \VV_h=\P_r\Lambda^0\x
\P_r^-\Lambda^1$ and the hypotheses required by \cite{bulletin} (the spaces
belong to a subcomplex of the Hilbert complex
$H^1\xrightarrow{\curl}\Hdiv\xrightarrow{\div}L^2$ with bounded cochain
projections) are satisfied.  From this it follows that the mixed finite element
method is stable and convergent.  Similar considerations apply to the magnetic
boundary value problem, where the finite element spaces are
$\0\Sigma_h=\Sigma_h\cap\0H^1$ and $\0\VV_h=\VV_h\cap\Hodiv$ and the relevant
Hilbert complex is $\0H^1\xrightarrow{\curl}\Hodiv\xrightarrow{\div}L^2$.
Another possible choice is to take $\Sigma_h$ to be Lagrange elements of
degree $r>1$ and $\VV_h$ to be Brezzi--Douglas--Marini elements of degree
$r-1$ (i.e., $\Sigma_h\x \VV_h=\P_r\Lambda^0\x \P_{r-1}\Lambda^1$).  This case
is similar, and will not be discussed further here.

We turn now to the main consideration of the current paper, which is the
equation \eqref{vec-Lap} with Dirichlet boundary conditions $\uu=0$ on
$\partial\Omega$.  This problem may of course be treated in the weak
formulation \eqref{weak-form} with $H=\0H^1(\Omega;\R^2)$.  In this case we
may integrate by parts and rewrite the bilinear form in terms of the gradient
(which, when applied to a vector, is matrix-valued):
$$
(\rot\uu ,\rot \vv) + (\div\uu,\div\vv) 
= (\grad\uu,\grad\vv), \quad \uu,\vv\in\0H^1(\Omega;\R^2).
$$
Thus the weak formulation \eqref{weak-form} is just
\begin{equation}
\label{gwf}
(\grad\uu,\grad\vv)=(\ff,\vv), \quad \vv\in\0H^1(\Omega;\R^2) ,
\end{equation}
for which the discretization using Lagrange or similar finite elements is
completely standard.

However, one might consider using a mixed method analogous to
\eqref{mm1}--\eqref{mm2} for the Dirichlet boundary value problem in the hope
of getting a better approximation of $\sigma=\rot\uu$, or
when Dirichlet boundary conditions are imposed on part of the
boundary and electric and/or magnetic boundary conditions are imposed on
another part of the boundary.  In fact, as we discuss in
Sections~\ref{sec:biharmonic} and \ref{sec:stokes}, a mixed approach to the
vector Laplacian with Dirichlet boundary conditions is implicitly used in
certain approaches to the solution of the Stokes equations which introduce the
vorticity, and in certain mixed methods for the biharmonic equation.  In the
mixed formulation of the Dirichlet problem for the vector Laplacian, the
vanishing of the normal component is an essential boundary condition, while
the vanishing of the tangential component arises as a natural boundary
condition.  No boundary conditions are imposed on the variable $\sigma$.
Thus, we define $\0\VV_h=\VV_h\cap \Hodiv$, and seek $\sigma_h \in \Sigma_h$,
$\uu_h \in \0\VV_h$ satisfying
\begin{gather}\label{md1}
(\sigma_h, \tau) - (\uu_h, \curl \tau) =0, \quad \tau \in \Sigma_h,
\\\label{md2}
(\curl \sigma_h, \vv) + (\div \uu_h, \div \vv) = (\ff,\vv), \quad
\vv \in \0\VV_h.
\end{gather}
Note that $\curl \Sigma_h\nsubseteq \0\VV_h$, so there is no Hilbert complex
available in this case, and the theory of \cite{bulletin} does not apply.
This suggests that there may be difficulties with stability and
convergence of the mixed method \eqref{md1}--\eqref{md2}.  In the next section,
we exhibit computational examples demonstrating that this pessimism is well
founded. The convergence of the mixed method for the Dirichlet boundary value
problem is severely suboptimal (while it is optimal for electric and magnetic
boundary conditions).  Thus, the difficulties arising from the loss of the
Hilbert complex structure are real, not an artifact of the theory.

However, the computations indicate that even for Dirichlet boundary
conditions, the mixed method does converge, albeit in a suboptimal manner.
While we do not recommend the mixed formulation for the Dirichlet problem,
in Section~\ref{sec:analysis} we prove convergence at the suboptimal rates that are observed
and, in so doing, clarify the sources of the suboptimality.
Theorem~\ref{thm:main}
summarizes the main results of our analysis, and the remainder 
of the section develops the tools needed to establish them.

This analysis of the
mixed finite element approximation of the vector Laplacian has implications
for the analysis of mixed methods for other important problems:
for the biharmonic equation using the
Ciarlet--Raviart mixed formulation, and for the Stokes
equations using a mixed formulation involving the vorticity, velocity, and
pressure, or, equivalently, using a stream function-vorticity formulation.
As a simple consequence of our analysis of the vector Laplacian, we
are able to analyze mixed methods for these problems, elucidating
the suboptimal rates of convergence observed for them, and establishing
convergence at the rates that do occur.  Some of the estimates we
obtain are already known, while others improve on existing
estimates. The biharmonic problem is addressed in
Section~\ref{sec:biharmonic} and the Stokes equations in Section~\ref{sec:stokes}.

We end this introduction with a summary of the main notations used in the
paper. For sufficiently smooth scalar-valued and vector-valued functions
$\sigma$ and $\uu$, respectively, we use the standard calculus operators
$$
\grad\sigma=(\pd{\sigma}x,\pd{\sigma}y), \quad
\curl\sigma=(\pd{\sigma}y,-\pd{\sigma}x), \quad
\div\uu = \pd{u_1}x+\pd{u_2}y,\quad
\rot\uu = \pd{u_2}x-\pd{u_1}y.
$$
We use the standard Lebesgue and Sobolev spaces $L^p(\Omega)$,
$H^l(\Omega)$, $W^l_p(\Omega)$, and also the spaces $H(\div,\Omega)$ and
$H(\rot,\Omega)$ consisting of $L^2$ vector fields $\uu$ with $\div\uu$ in
$L^2$ or $\rot\uu\in L^2$, respectively.  Since the domain $\Omega$ will
usually be clear from context, we will abbreviate these spaces as $L^p$,
$H^l$, $\Hdiv$, etc.  For vector-valued functions in a Lebesgue or Sobolev
space, we may use notations like $H^l(\Omega;\R^2)$, although when there is
little chance of confusion we will abbreviate this to simply $H^l$.  The
closure of $C^\infty_0(\Omega)$ in $H^1$, $\Hdiv$, and $\Hrot$, are denoted
$\0H^1$, $\Hodiv$, $\Horot$.  Note that if $\uu\in\Hdiv$, then the normal
trace $\uu\cdot\nn\in H^{-1/2}(\partial\Omega)$ and $\Hodiv =
\{\,\uu\in\Hdiv\,|\, \uu\cdot\nn=0 \text{ on $\partial\Omega$}\,\}$.
Similarly, $\Horot = \{\,\uu\in\Hrot\,|\, \uu\cdot\ssb=0 \text{ on
$\partial\Omega$}\,\}$.  We write $(\cdot,\cdot)$ for the $L^2(\Omega)$ inner
product (of either scalar- or vector-valued functions) and $\|\cdot\|$ for the
corresponding norm.

We shall also need the dual space of $\Hodiv$, the space $\Hodiv'$, 
normed by
\begin{equation}\label{dn}
\|\vv\|_{\Hodiv'}:= \sup_{\ww\in \Hodiv}\frac{(\vv,\ww)}
{\|\ww\|_{H(\div)}}.
\end{equation}  
Clearly,
\begin{equation}
  \label{eq:incl}
  L^2(\Omega;\R^2)\subset \Hodiv'\subset H^{-1}(\Omega;\R^2)
\end{equation}
with continuous inclusions.  

\section{Some numerical results}\label{sec:numerics}

We begin by considering the solution of the Hodge Laplacian \eqref{vec-Lap}
with electric boundary conditions \eqref{ebc} using the mixed method
\eqref{mm1}, \eqref{mm2}.  For the space $\Sigma_h$, we use Lagrange finite
elements of degree $r\ge1$ and for the space $V_h$, Raviart--Thomas elements
of degree $r$ (consisting locally of certain polynomials of degree $\le r$,
including all those of degree $\le r-1$).  These are stable elements and a
complete analysis is given in \cite{bulletin}.  Assuming that the solution is
smooth, it follows from \cite[Theorem~3.11]{bulletin} that the following rates
of convergence, each optimal, hold:
$$
\|u-u_h\| = O(h^r), \quad
\|\div(u-u_h)\| = O(h^r), \quad
\|\sigma-\sigma_h\| = O(h^{r+1}), \quad
\|\grad(\sigma-\sigma_h)\| = O(h^r).
$$
Table~\ref{tb:t1} shows the results of a computation with $r=2$.  Note that
the computed rates of convergence are precisely as expected.  In the test
problem displayed, the domain is $\Omega=(0,1)\x(0,1)$ and the exact solution
is $\uu = (\cos\pi x\sin\pi y,2\sin\pi x\cos\pi y)$.  The meshes used for
computation were obtained by dividing the square into $2^n\x2^n$ subsquares,
$n=1,2,4,\ldots 128$, and dividing each subsquare into two triangles with the
positively sloped diagonal.  Only the result for the four finest meshes are
shown.  Very similar results were obtained for the case of magnetic boundary
conditions, and for a sequence of nonuniform meshes, and also for other values
of $r\ge 1$.

%Table 1
\begin{table}[!ht]
\footnotesize
\caption{$L^2$ errors and convergence rates for degree $2$ mixed finite element
approximation of the vector Laplacian with electric boundary conditions.}
\label{tb:t1}
\begin{center}
\begin{tabular}{>{\centering}p{.85in}c|>{\centering}p{.85in}c|>{\centering}p{.85in}c|>{\centering}p{.85in}c}%{cc|cc|cc|cc}
  $\|\uu-\uu_h\|$    &     rate  &    $\|\div(\uu-\uu_h)\|$  &   rate  &    
$\|\sigma-\sigma_h\|$  &   rate  &   $\|\curl(\sigma-\sigma_h)\|$  &  rate \\[.05in]
\hline\vrule height.2in width0in
 2.14e-03 & 1.99 & 1.17e-02 & 1.99 & 2.16e-04 & 3.03 & 2.63e-02 & 1.98\\
 5.37e-04 & 1.99 & 2.93e-03 & 2.00 & 2.70e-05 & 3.00 & 6.60e-03 & 1.99\\
 1.34e-04 & 2.00 & 7.33e-04 & 2.00 & 3.37e-06 & 3.00 & 1.65e-03 & 2.00\\
 3.36e-05 & 2.00 & 1.83e-04 & 2.00 & 4.16e-07 & 3.02 & 4.14e-04 & 2.00\\
\end{tabular}
\end{center}
\end{table}

The situation in the case of Dirichlet boundary conditions is very
different. In Table~\ref{tb:t2} we consider the problem with exact solution
$\uu=(\sin\pi x\sin\pi y,\sin\pi x\sin\pi y)$.  The finite element spaces are
as for the computation of Table~\ref{tb:t1}, except that the boundary condition
of vanishing normal trace is imposed in the Raviart--Thomas space $V_h$.  Note
that the $L^2$ rate of convergence for $\sigma$ is not the optimal value of
$3$, but rather roughly $3/2$.  The $L^2$ rate of convergence of $\curl\sigma$
(i.e., the $H^1$ rate of convergence of $\sigma$) is also suboptimal by
roughly $3/2$: it converges only as $h^{1/2}$.  For $\uu$, the $L^2$
convergence rate is the optimal $2$, but the convergence rate for $\div\uu$ is
suboptimal by $1/2$.

%Table 2
\begin{table}[!ht]
\footnotesize
\caption{$L^2$ errors and convergence rates for degree $2$ mixed finite element
approximation of the vector Laplacian with Dirichlet boundary conditions.}
\label{tb:t2}
\begin{center}
\begin{tabular}{>{\centering}p{.85in}c|>{\centering}p{.85in}c|>{\centering}p{.85in}c|>{\centering}p{.85in}c}%{cc|cc|cc|cc}
  $\|\uu-\uu_h\|$    &     rate  &    $\|\div(\uu-\uu_h)\|$  &   rate  &    
$\|\sigma-\sigma_h\|$  &   rate  &   $\|\curl(\sigma-\sigma_h)\|$  &  rate \\[.05in]
\hline\vrule height.2in width0in
1.22e-03 & 2.01 & 1.55e-02 & 1.58 & 1.90e-02 & 1.62 & 2.53e+00 & 0.63 \\
3.05e-04 & 2.00 & 5.33e-03 & 1.54 & 6.36e-03 & 1.58 & 1.68e+00 & 0.60 \\
7.63e-05 & 2.00 & 1.85e-03 & 1.52 & 2.18e-03 & 1.54 & 1.14e+00 & 0.56 \\
1.91e-05 & 2.00 & 6.49e-04 & 1.51 & 7.58e-04 & 1.52 & 7.89e-01 & 0.53 \\
\end{tabular}
\end{center}
\end{table}

We have carried out similar computations for $r=3$ and $4$ and for nonuniform
meshes and the results are all very similar: degradation of the rate of
convergence by $3/2$ for $\sigma$ and $\curl\sigma$, and by $1/2$ for $\div
u$.  However the case $r=1$ is different.  There we saw no degradation of
convergence rates for uniform meshes, but for nonuniform meshes $\sigma$
converged in $L^2$ with rate suboptimal by $1$ and $\curl\sigma$ did not
converge at all.

The moral of this story is that the mixed finite element method using the
standard elements is indeed strongly tied to the underlying Hilbert complex
structure which is not present for the vector Laplacian with Dirichlet
boundary conditions, and the method is not appropriate for this problem.
Nonetheless the experiments suggest that the method does converge, albeit at a
degraded rate.  In the next section, we develop the theory needed to prove that
this is indeed so, and also to indicate where the lack of Hilbert complex
structure leads to the suboptimality of the method.

\section{Error analysis}\label{sec:analysis}
Theorem~\ref{thm:main}, which is the primary result of this section,
establishes convergence of the mixed method for the Dirichlet problem at the
suboptimal rates observed in the previous section.  In it we suppose that
$\Omega$ is a convex polygon endowed with a shape-regular and quasi-uniform
family of triangulations
of mesh size $h$.
We continue to denote by $\Sigma_h\subset H^1$ and
$\VV_h\subset\Hdiv$ the Lagrange and Raviart--Thomas finite element spaces of
some fixed degree $r \ge 1$, respectively, with $\0\VV_h=\VV_h\cap \Hodiv$.
\begin{thm}\label{thm:main}
Let $\uu$ denote the solution of the vector Laplace equation \eqref{vec-Lap}
with Dirichlet boundary condition $\uu=0$, and let $\sigma=\rot\uu$.  There
exist unique $\sigma_h\in\Sigma_h$, $\uu_h\in \0\VV_h$ satisfying the mixed
method \eqref{md1}--\eqref{md2}.  If the polynomial degree $r\ge 2$, then
the following estimates hold for $2 \le l \le r$ (whenever the norms on the
right hand side are finite):
\begin{gather*}
 \|\uu-\uu_h\| \le C h^l\|\uu\|_{l},
\\
\|\div(\uu-\uu_h)\|+ \|\sigma-\sigma_h\| 
+ h \|\curl(\sigma-\sigma_h)\| \le
C h^{l-1/2}(|\ln h| \|\uu\|_{W^l_{\infty}} + \|\uu\|_{l+1/2}).
\end{gather*}
If $r=1$, the estimates are:
\begin{gather*}
 \|\uu-\uu_h\| \le Ch |\ln h|^2(|\ln h| \|\uu\|_{W^1_{\infty}} + \|\uu\|_2),
\\
\|\div(\uu-\uu_h)\|+ \|\sigma-\sigma_h\| 
+ h \|\curl(\sigma-\sigma_h)\| \le
C h^{1/2} (|\ln h| \|\uu\|_{W^1_{\infty}} + h^{1/2} \|\uu\|_2).
\end{gather*}
\end{thm}
Note that the error estimate for $\uu$ is optimal order (modulo the logarithm
when $r=1$), while (again modulo the logarithm), the estimate for $\div u$ is
suboptimal by $1/2$ order, and the estimates for $\sigma$ and $\curl\sigma$
are suboptimal by $3/2$ order.  This is as observed in the experiments
reported above. Above and throughout, we use $C$ to denote a generic constant
independent of $h$, whose values may differ at different occurrences.

The proof of this theorem is rather involved.  Without the Hilbert complex
structure, the numerical method is not only less accurate, but also harder to
analyze.  The analysis will proceed in several steps.  First, in
Section~\ref{wp}, we will establish the well-posedness of the continuous
problem, not in the space $H^1\x \Hodiv$, but rather using a larger space than
$H^1$ with weaker norm for $\sigma$.  Next, in Section~\ref{wp-d}, we mimic
the well-posedness proof on the discrete level to obtain stability of the
discrete problem, but with a mesh-dependent norm on $\Sigma_h$.  This norm is
even weaker than the norm used for the continuous problem, which may be seen
as the cause of the loss of accuracy.  To continue the analysis, we then
introduce projection operators into $\0V_h$ and $\Sigma_h$ and develop bounds
and error estimates for them in Section~\ref{projections}.  In
Section~\ref{basic} we combine these with the stability result to obtain basic
error estimates for the scheme, and we improve the error estimate for $\uu_h$
in Section~\ref{improved} using duality.

\subsection{Preliminaries}    \label{prelim}
First we recall two forms of the Poincar\'e--Friedrichs inequality:
\begin{equation}\label{pf}
\|\tau\|\le C_P\|\curl\tau\|, \quad \tau\in\0H^1, \qquad
\|\psi\|\le C_P\|\grad\psi\|,\quad \psi\in \hat H^1.
\end{equation}
Here $\hat H^1$ denotes the subspace of functions in $H^1$ with
zero mean. Similarly, we will use $\hat L^2$ to denote the zero mean
subspace of $L^2$.

Next we recall the Hodge decomposition.  The space $L^2(\Omega;\R^2)$ admits a
decomposition into the orthogonal closed subspaces $\curl H^1$ and $\grad
\0H^1$, or, alternatively, into the subspaces $\curl\0H^1$ and $\grad H^1$.
The decomposition of a given $\vv\in L^2$ according to either of these may be
computed by solving appropriate boundary value problems.  For example, we may
compute the unique $\rho\in\0H^1$ and $\phi\in\hat H^1$ such that
\begin{equation}
  \label{eq:1}
\vv = \curl\rho + \grad \phi,  
\end{equation}
by a Dirichlet problem and a Neumann problem for the scalar Poisson 
equation, respectively:
\begin{equation}\label{bvps}
\begin{gathered}
(\curl\rho,\curl\tau)=(\vv,\curl\tau),\quad \tau\in\0H^1,\\
(\grad\phi,\grad\psi)=(\vv,\grad\psi),\quad \psi\in\hat H^1.
\end{gathered}
\end{equation}
Clearly, $\|\grad\phi\|\le \|\vv\|$.  If $\vv\in\Hodiv$, then
$\phi$ satisfies the Neumann problem
$$
\Delta\phi=\div\vv \text{ in $\Omega$}, \quad
\frac{\partial\phi}{\partial n}=0 \text{ on $\partial\Omega$},\quad
\int_\Omega\phi\,dx=0,
$$
so, by elliptic regularity, $\|\phi\|_2\le C\|\div\vv\|$ if the domain
is convex, and $\|\phi\|_1\le C\|\div\vv\|$ for any domain.

We shall need analogous results on the discrete level.  To this end,
let $S_h$ denote the space of piecewise polynomials of degree at most $r-1$,
with no imposed interelement continuity.  Then the divergence operator maps
$\VV_h$ onto $S_h$ and also maps $\0\VV_h$ onto $\hat S_h$, the codimension
one subspace consisting of functions with mean value zero.  The former
pair of spaces is used to solve the Dirichlet problem for the Poisson
equation, and the later is used to solve the Neumann problem.
Each pair forms part of a short exact sequence:
\begin{equation}\label{seqs}
0\to \hat\Sigma_h \xrightarrow{\curl} \VV_h \xrightarrow{\div} S_h \to 0
\text{\quad and \quad}
0\to \0\Sigma_h \xrightarrow{\curl} \0\VV_h \xrightarrow{\div} \hat S_h \to 0,
\end{equation}
respectively.

The usual Raviart--Thomas approximate solution to Poisson equation $\Delta
\phi=g$ with Dirichlet boundary condition $\phi=0$ is then: find $\vv_h\in
\VV_h$, $\phi_h\in S_h$ such that
\begin{gather*}
 (\vv_h,\ww)+(\div\ww,\phi_h)=0,\quad \ww\in\VV_h,\qquad
(\div\vv_h,\psi)=(g,\psi), \quad \psi\in S_h.
\end{gather*}
Define the operator $\gradh:S_h\to\VV_h$ by
$$
(\gradh\phi,\ww)=-(\phi,\div\ww), \quad \phi\in S_h,\ \ww\in\VV_h.
$$
From the stability of the mixed method, we obtain the discrete Poincar\'e
inequality
$\|\phi\|\le \bar C_P\|\gradh\phi\|$, $\phi\in S_h$, 
with $\bar C_P$ independent of $h$.
The solution $(\vv_h,\phi_h)\in\VV_h\x S_h$ of the mixed method may be 
characterized by
$$
(\gradh\phi_h,\gradh\psi)= - (g,\psi),\quad\psi\in S_h,
$$
and $\vv_h=\gradh\phi_h$.

Corresponding to the first sequence in \eqref{seqs}, we have
the discrete Hodge decomposition 
\begin{equation}
  \label{eq:Hodge}
   \VV_h = \curl \Sigma_h + \gradh S_h,
\end{equation}
and corresponding to the second, the alternate discrete Hodge
decomposition 
\begin{equation}
  \label{eq:HodgeA}
\0\VV_h =\curl\0\Sigma_h+\gradoh S_h,
\end{equation}
where  
$\gradoh:S_h\to\0\VV_h$ is defined by
$$
(\gradoh\phi,\ww)=-(\phi,\div\ww), \quad \phi\in S_h,\ \ww\in\0\VV_h.
$$
Both of the discrete Hodge decompositions can be characterized
by finite element computations.  For example, in analogy to \eqref{bvps},
for given $\vv\in \0\VV_h$ we may compute the unique
$\rho_h\in \0\Sigma_h$ and $\phi_h\in\hat S_h$ such that 
$\vv=\curl\rho_h+\gradoh\phi_h$
from the following finite element systems (one primal, one mixed):
\begin{gather*}
(\curl\rho_h,\curl\tau)=(\vv,\curl\tau),\quad \tau\in\0\Sigma_h,\\
(\gradoh\phi_h,\gradoh\psi)=(\vv,\gradoh\psi),\quad \psi\in\hat S_h.
\end{gather*}

\subsection{Well-posedness of the continuous formulation}
\label{wp}

As a first step towards analyzing the mixed method, we obtain well-posedness
of a mixed formulation of the continuous boundary value problem for the vector
Laplacian. To do so, we need to introduce a larger space than $H^1$ for the
scalar variable, namely
$$
\Sigma = \{\tau \in L^2: \curl \tau \in 
\Hodiv'\},
$$
with norm $\|\tau\|_{\Sigma}^2 =
\|\tau\|^2 + \|\curl \tau\|_{\Hodiv'}^2$ (see \eqref{dn}).
The space $\Sigma$ has appeared before
in studies of the vorticity-velocity-pressure and stream
function-vorticity formulations of the Stokes problem \cite{d-s-s-2},
and an equivalent space (at least for domains with $C^{1,1}$ boundary)
is used in \cite{b-g-m}.
The bilinear form for the mixed formulation is
\begin{equation*}
B(\rho, \ww; \tau, \vv) =
(\rho, \tau)
- \langle \curl \tau,\ww \rangle + \langle\curl \rho, \vv\rangle
+ (\div \ww, \div \vv),
\end{equation*}
where $\langle \cdot,\cdot\rangle$ denotes the pairing between
  $\Hodiv'$ and $\Hodiv$ (or more generally between a Hilbert space
and its dual.)
Often, we will tacitly use the fact that if
  $\tau$ is in $H^1$, then $ \langle \curl \tau,\ww \rangle = ( \curl
  \tau, \ww)$.  Clearly,
\begin{equation*}
|B(\rho, \ww; \tau, \vv)|
\le 2(\|\rho\|_{\Sigma}^2 + \|\ww\|_{\Hdiv}^2)^{1/2}
(\|\tau\|_{\Sigma}^2 + \|\vv\|_{\Hdiv}^2)^{1/2},\quad
\rho,\tau\in\Sigma,\ \ww,\vv\in\Hodiv,
\end{equation*}
so $B$ is bounded on $(\Sigma\times \Hodiv)
\times (\Sigma\times \Hodiv)$.
For $\tau\in\Sigma$, we define $\tau_0\in\0H^1$ by
$$
(\curl\tau_0,\curl\psi)=\<\curl\tau,\curl\psi\>, \quad \psi\in \0H^1,
$$
Taking $\psi=\tau_0$ shows that
\begin{equation}\label{ct0}
\|\curl\tau_0\|\le \|\curl\tau\|_{\Hodiv'}
\le \|\tau\|_\Sigma,\quad \tau\in\Sigma.
\end{equation}
It is also true that
\begin{equation}\label{ct}
 \|\tau\|_{\Sigma}\le C(\|\tau\|+\|\curl\tau_0\|),\quad \tau\in\Sigma.
\end{equation}
To see this, define $\phi\in\hat L^2$ by 
\begin{equation}\label{phidef}
(\phi,\div\vv) = \<\curl\tau,\vv\> -(\curl\tau_0,\vv), \quad \vv\in\Hodiv.
\end{equation}
This is well-defined, since $\div\Hodiv=\hat L^2$, and, if $\div\vv$ vanishes,
then $\vv=\curl\psi$ for some $\psi\in\0H^1$, so the right-hand side vanishes
as well.  Clearly,
$$                                                                   
\<\curl\tau,\vv\> = (\curl\tau_0,\vv)+(\phi,\div\vv)
\le (\|\curl\tau_0\|+\|\phi\|)\|\vv\|_{\Hdiv},\quad \vv\in\Hodiv.
$$
Choosing $\vv\in\0H^1$ in \eqref{phidef}
with $\div\vv=\phi$ and $\|\vv\|_1\le C\|\phi\|$,
we get
$\|\phi\|\le C\|\curl(\tau-\tau_0)\|_{-1}$.
This implies
 $\|\curl\tau\|_{\Hodiv'}\le C(\|\tau\|+\|\curl\tau_0\|)$, thus
establishing \eqref{ct}.  We conclude from~\eqref{ct0} and~\eqref{ct} that
 the norm $\tau\mapsto \|\tau\|+\|\curl\tau_0\|$ is an 
equivalent norm on~$\Sigma$.

Assuming that $\ff\in L^2$ (or even $\Hodiv'$), we now
give a mixed variational formulation of the continuous problem.  We seek $\sigma \in
\Sigma$, $\uu \in \Hodiv$, 
such that
\begin{align*}
(\sigma, \tau) - \langle \curl \tau,\uu \rangle
& =0, \quad \tau \in \Sigma,
\\
\langle \curl \sigma, \vv \rangle
+ (\div \uu, \div \vv)
& = (\ff,\vv), \quad
\vv \in \Hodiv.
\end{align*}
We note that, if $\uu\in\0 H^1(\Omega; \R^2)$ is the solution
of the standard variational formulation \eqref{gwf} and $\sigma=\rot\uu$,
then $\sigma$, $\uu$ solve this mixed variational formulation.  Indeed,
$\uu\in\0 H^1(\Omega; \R^2)\subset\Hodiv$, $\sigma\in L^2$, and, for $\vv\in\0 H^1(\Omega; \R^2)$
$$
\<\curl\sigma,\vv\> = (\sigma,\rot\vv)=(\rot\uu,\rot\vv)=(\ff,\vv)-(\div\uu,\div\vv).
$$
This implies that $\curl\sigma\in\Hodiv'$, so $\sigma\in\Sigma$, and, extending to $\vv\in\Hodiv$
by density, that the second equation above holds.  Finally
$$
(\sigma,\tau)=(\rot\uu,\tau)=\<\uu,\curl\tau\>
$$
for all $\tau\in L^2$, so the first equation holds.

In the next theorem, we establish well-posedness of the mixed variational problem by proving the
inf-sup condition for $B$, following the approach of \cite{acta}.  Note
that the theorem establishes well-posedness of the more general problem where the
zero on the right hand side of the first equation is replaced by the linear
functional $\langle g, \tau \rangle$, where $g \in \Sigma^{\prime}$, and we
allow $\ff \in \Hodiv^{\prime}$.
\begin{thm}
\label{well-p}
There exist constants $c >0$, $C < \infty$ such that, for any
$(\rho, \ww) \in \Sigma \times \Hodiv$, there exists
$(\tau, \vv) \in \Sigma \times \Hodiv$ with
\begin{align}
\label{coercivity}
B(\rho, \ww; \tau, \vv) &\ge c 
(\|\rho\|_{\Sigma}^2 + \|\ww\|_{\Hdiv}^2),
\\
\label{bound}
\|\tau\|_{\Sigma} + \|\vv\|_{\Hdiv} &\le 
C(\|\rho\|_{\Sigma} + \|\ww\|_{\Hdiv}).
\end{align}
Moreover, if $\ww\in\curl \0H^1$, then we may choose $\vv\in\curl\0H^1$.
\end{thm}
\begin{proof}
Define $\rho_0 \in \0H^1$ by $(\curl \rho_0, \curl \psi) =
\<\curl \rho, \curl \psi\>$, $\psi \in \0H^1$.
Next, use the Hodge decomposition to write 
$\ww$ in the form
$\ww = \curl \mu + \grad \phi$, with $\mu \in \0 H^1$
and $\phi \in \hat H^1$, and recall that
\begin{equation}\label{t}
 \|\grad\phi\|\le C\|\div \ww\|.
\end{equation}
We then choose
\begin{equation*}
\tau = \rho - \delta \mu, \qquad
\vv = \ww + \curl \rho_0,
\end{equation*}
where $\delta$ is a constant to be chosen.
Hence, 
\begin{align*}
B(\rho, \ww; \tau, \vv) &=
\|\rho\|^2 - \delta (\rho, \mu) - \langle \curl \rho, \ww\rangle
+ \delta (\curl \mu, \ww)
\\
&\qquad + \langle \curl \rho, \ww \rangle
 +\langle \curl \rho, \curl \rho_0\rangle
+  \|\div \ww\|^2
\\
&= \|\rho\|^2 + \delta \|\curl \mu\|^2  - \delta (\rho, \mu)  
 + \|\curl \rho_0\|^2  + \|\div \ww\|^2.
\end{align*}
Recalling the constant $C_P$ in the Poincar\'e inequality \eqref{pf}
and choosing $\delta$ sufficiently small, we obtain
\begin{align*}
B(\rho, \ww; \tau, \vv)
& \ge \frac{1}{2} \|\rho\|^2 + (\delta- \delta^2 C_P^2/2)\|\curl \mu\|^2 
 + \|\curl \rho_0\|^2 + \|\div \ww\|^2
\\
& \ge c\big( \|\rho\|_{\Sigma}^2 + \|\ww\|_{\Hdiv}^2\big),
\end{align*}
where we have used the facts that
$\|\ww\|^2 = \|\curl \mu\|^2 + \|\grad \phi\|^2$, \eqref{t}, and \eqref{ct}
in the last step.  This establishes \eqref{coercivity}.

To establish \eqref{bound}, we observe that
\begin{equation*}
\|\vv\|_{\Hdiv} \le \|\ww\|_{\Hdiv} + \|\curl \rho_0\|
\le \|\ww\|_{\Hdiv} + \|\rho\|_{\Sigma}
\end{equation*}
by \eqref{ct0}, while
\begin{equation*}
\|\tau\|_{\Sigma} \le \|\rho\|_{\Sigma} + \delta \|\mu\|_{\Sigma} \le
\|\rho\|_{\Sigma} + \delta \|\mu\|_1
\le C(\|\rho\|_{\Sigma} + \|\ww\|),
\end{equation*}
since $\|\mu\|_1\le C \|\curl\mu\|\le C\|\ww\|$.

To establish the final claim, we observe that if 
$\ww \in \curl \0H^1$, then obviously
$\vv = \ww + \curl \rho_0 \in \curl \0H^1$.
%Moreover,
%\begin{multline*}
%\|\curl \tau_0\|^2 = (\curl \tau, \curl \tau_0)
%= (\curl \sigma, \curl \tau_0) + t (\curl \rho, \curl \tau_0)
%\\
%= (\curl \sigma_0, \curl \tau_0) + t (\curl \rho, \curl \tau_0)
%\le C[\|\curl \sigma_0\| + \|\uu\|]\|\curl \tau_0\|.
%\end{multline*}
%Hence,
%\begin{equation*}
%\|\curl \tau_0\| \le C[\|\curl \sigma_0\| + \|\uu\|].
%\end{equation*}
%Combining these results, we get the upper bound \eqref{bound}, i.e.,
%\begin{equation*}
%\|\tau\| + \|\curl \tau_0\|  + \|v\|_{\Hdiv}
%\le C[\|\sigma\| + \|\curl \sigma_0\|  + \|\uu\|_{\Hdiv}]
%\le C[\|\sigma\|_{\Sigma} + \|\uu\|_{\Hdiv}].
%\end{equation*}
\end{proof}

\begin{remark}
Had we posed the weak formulation using the space $H^1\times
\Hodiv$ instead of $\Sigma\times \Hodiv$, we would not have obtained a
well-posed problem.
\end{remark}

\subsection{Stability of the discrete formulation}
\label{wp-d}

In this section, we establish the stability of the mixed method
\eqref{md1}--\eqref{md2}, guided by the arguments used for the continuous
problem in the preceding subsection.  Analogous to the norm on
$\Sigma$, we begin by defining a norm on $\Sigma_h$ by
$\|\tau\|_{\Sigma_h}^2 =
\|\tau\|^2 + \|\curl\tau\|_{\0\VV_h'}^2$, $\tau\in \Sigma_h$, where
\begin{equation*}
\|\vv\|_{\0 \VV_h'}:= \sup_{\ww\in \0\VV_h}\frac{(\vv,\ww)}{\|\ww\|_{H(\div)}}.
\end{equation*}
The bilinear form is bounded on the finite element spaces in this norm:
\begin{equation*}
|B(\rho, \ww; \tau, \vv)|
\le 2(\|\rho\|_{\Sigma_h}^2 + \|\ww\|_{\Hdiv}^2)^{1/2}
(\|\tau\|_{\Sigma_h}^2 + \|\vv\|_{\Hdiv}^2)^{1/2}, \quad
\rho, \tau \in \Sigma_h, \ \ww, \vv \in \0{\VV}_h.
\end{equation*}
For $\tau\in\Sigma_h$, we define
$\tau_0 \in \0 \Sigma_h$ by
\begin{equation*}
(\curl \tau_0, \curl \psi) = (\curl \tau, \curl \psi), \quad \psi \in 
\0 \Sigma_h.
\end{equation*}
The discrete analogue of \eqref{ct0} again follows by choosing
$\psi=\curl\tau_0$:
\begin{equation*}
\|\curl\tau_0\|\le \|\curl\tau\|_{\0V_h'}\le \|\tau\|_{\Sigma_h}, 
\quad \tau\in\Sigma_h.
\end{equation*}
Next we establish discrete analogue of \eqref{ct}, that is,
\begin{equation}\label{cth}
 \|\tau\|_{\Sigma_h}\le C(\|\tau\|+\|\curl\tau_0\|),\quad \tau\in\Sigma_h.
\end{equation}
To see this, define $\phi\in\hat S_h$ by
$$
(\phi,\div\vv) = (\curl\tau,\vv) - (\curl\tau_0,\vv), \quad \vv\in\0\VV_h.
$$
This is well-defined, since $\div\0\VV_h=\hat S_h$, and, if $\div\vv$ vanishes,
then $\vv=\curl\psi$ for some $\psi\in\0\Sigma_h$, so the right-hand side 
vanishes as well.  It follows that
$\|\curl\tau\|_{\0\VV_h'} \le \|\curl\tau_0\|+\|\phi\|$.
To bound $\|\phi\|$,
as in the continuous case, we choose $\vv\in\0H^1$ with
$\div\vv=\phi$ and $\|\vv\|_1\le C\|\phi\|$.  In the discrete case, 
we also introduce
$\PiV\vv$, the canonical projection of $\vv$ into the Raviart--Thomas space
$\0\VV_h$ (see \eqref{eq:dofs}), so $\div\PiV\vv=P_{S_h}\div\vv=\phi$ 
and $\|\vv-\PiV\vv\|\le C h\|\vv\|_1$.
Then
\begin{align*}
 \|\phi\|^2 &= (\phi,\div\PiV\vv) =(\curl\tau,\PiV\vv)-(\curl\tau_0,\PiV\vv)
\\
&= (\curl\tau,\PiV\vv-\vv) + (\curl\tau,\vv) -(\curl\tau_0,\PiV\vv)
\\
&\le C h (\|\tau\|_1 +\|\curl\tau\|_{-1}+\|\curl\tau_0\|)\|\vv\|_1.
\end{align*}
Using the inverse inequality $\|\tau\|_1\le Ch^{-1}\|\tau\|$ and the fact
that $\|\vv\|_1\le \|\phi\|$, gives the bound
$\|\phi\|\le C(\|\tau\|+\|\curl\tau_0\|)$,
and implies \eqref{cth}.

With this choice of norm, stability of the finite element approximation scheme
is established by an argument precisely analogous to that used in the proof of
Theorem~\ref{well-p}, simply using the $\Sigma_h$ norm, the discrete gradient
operator $\gradoh$, the discrete Hodge decomposition
\eqref{eq:HodgeA}, the estimate~\eqref{cth}, and the discrete
Poincar\'e inequality, instead of their continuous counterparts.
\begin{thm}
\label{well-pd}
There exists constants $c >0$, $C < \infty$, independent of $h$, such
that, for any $(\rho, \ww) \in \Sigma_h \times \0{\VV}_h$, there
exists $(\tau, \vv) \in \Sigma_h \times \0{\VV}_h$ with
\begin{align*}
B(\rho, \ww; \tau, \vv) &\ge c 
(\|\rho\|_{\Sigma_h}^2 + \|\ww\|_{\Hdiv}^2),
\\
\|\tau\|_{\Sigma_h} + \|\vv\|_{\Hdiv} &\le 
C(\|\rho\|_{\Sigma_h} + \|\ww\|_{\Hdiv}).
\end{align*}
Moreover, if $\ww\in\curl \0\Sigma_h$, then we may choose 
$\vv\in\curl\0\Sigma_h$.
\end{thm}

\begin{remark}
Note that $\|\tau\|_{\Sigma_h}\le \|\tau\|_\Sigma$ for $\tau\in\Sigma_h$,
but, in general equality does not hold.  Had we used the $\Sigma$ norm
instead of the $\Sigma_h$ norm on the discrete level, we would not have
been able to establish stability.%  The difference between these two norms
%may be viewed as the source of loss of accuracy.
\end{remark}

\subsection{Projectors}
\label{projections}

% Since the aim of the error analysis is to verify the results obtained
% in the computations and indicate the source of the loss of accuracy,
% we consider here only the case when $\Omega$ is a convex polygon,
% $\Sigma_h$ is the space of continuous piecewise polynomials of degree
% $\le k+1$, $\VV_h$ is the Raviart-Thomas space of degree $k+1$,
% and $S_h$ is the space of discontinuous, piecewise polynomials of
% degree $\le k$.  We will assume that the solution $(\sigma, \uu)$ of
% the mixed formulation has at least the minimum regularity $\sigma \in
% H^1$, $\uu \in H^2$. As discussed above, this
% regularity implies the equivalence of the mixed fomulation and
% formulation \eqref{svf}.
% \COMMENT{Probably we should remove the above paragraph.}

Our error analysis will be based on the approximation and orthogonality
properties of certain projection operators into the finite element spaces:
$$
\PS:L^2\to S_h, \quad \PSig:H^1\to\Sigma_h, \quad 
\PSigo:\0H^1\to\0\Sigma_h,  \quad \PV:\Hodiv\to\0\VV_h.
$$
For $\PS$, we simply take the $L^2$ projection.  By standard approximation
theory,
$$
\|s-\PS s\|_{L^p}\le C h^l\|s\|_{W^l_p}, \quad 0\le l\le r,\ 1\le p\le\infty.
$$
%From this we also obtain, by duality, that
%\begin{equation}\label{eq:L2-dual}
%\|s-\PS s\|_{(W^1_{q})^*}\le C h\|s-\PS s\|_{L^p}, \quad 1\le p\le\infty,
%\end{equation}
%with $q=(p-1)/p$ the conjugate exponent.
%\COMMENT{Check that we need \eqref{eq:L2-dual}.}

For $\PSig$ and $\PSigo$, we use elliptic projections.  Namely, 
for any $\tau\in H^1$,
\begin{equation*}
(\curl {\PSig{\tau}}, \curl \rho) = (\curl \tau, \curl \rho), 
\quad \rho \in \Sigma_h, \qquad ({\PSig{\tau}}, 1) = (\tau,1),
\end{equation*}
and, for any $\tau\in\0H^1$
\begin{equation*}
(\curl {\PSigo{\tau}}, \curl \rho) = (\curl \tau, \curl \rho), 
\quad \rho \in \0\Sigma_h.
\end{equation*}
Then, by standard estimates,
\begin{equation}\label{eq:PS-err}
 \|\sigma- {\PSig{\sigma}}\| + h \|\sigma- {\PSig{\sigma}}\|_1\le
Ch^l\|\sigma\|_l, \quad 1\le l \le r+1.
\end{equation}
Moreover,
\begin{equation}\label{eq:curlterm}
    (\curl[\sigma-{\PSig{\sigma}}], \vv) 
    \le C h \|\curl (\sigma-{\PSig{\sigma}} )\|
    \|\div \vv\|, \quad \vv \in V_h, \;\sigma \in H^1.
\end{equation}
To prove this last estimate, we use
the discrete Hodge decomposition~\eqref{eq:Hodge}
to write $\vv = \curl \gamma_h + \grad_h
\psi_h$, with $\gamma_h \in \hat \Sigma_h$ and $\psi_h \in S_h$. As
explained in \S~\ref{prelim}, the pair $(\grad_h\psi_h, \psi_h)\in
V_h\x S_h$ is the mixed approximation of $(\grad\psi,\psi)$ where
$\psi\in \0H^1$ solves $\Lap \psi = \div \vv$ in
$\Omega$. Since $\Omega$ is convex, $\|\psi\|_2 \le C \|\div
\vv\|$.  Therefore,
\begin{align*}
(\curl[\sigma-{\PSig{\sigma}}], \vv)
&= (\curl[\sigma-{\PSig{\sigma}}], \curl \gamma_h + \grad_h \psi_h)
= (\curl[\sigma-{\PSig{\sigma}}], \grad_h \psi_h)
\\
&= (\curl[\sigma-{\PSig{\sigma}}], \grad_h \psi_h - \grad \psi)
\le C h\|\curl(\sigma-{\PSig{\sigma}})\| \|\psi\|_2
\\
& 
\le C h\|\curl(\sigma-{\PSig{\sigma}})\| \|\div \vv\|.
\end{align*}
For $\PSigo\tau$, $\tau\in\0H^1$, we will use the $W^1_p$ estimate
(due to Nitsche \cite{nitsche} for $r\ge 2$ and
Rannacher and Scott \cite{rannacher-scott} for $r=1$; cf.~also  
\cite[Theorem~8.5.3]{brenner-scott}):
\begin{equation}\label{w1p}
 \|\tau-\PSigo\tau\|_{W^1_p} \le C h^{l-1}\|\tau\|_{W^l_p}, \quad
1\le l \le r+1, \; 2\le p\le \infty,
\end{equation}
which holds with constant $C$ independent of $p$ as well as $h$.

We define the fourth projection operator, $\PV:\Hodiv\to\0\VV_h$,
by the equations
$$
(\PV\vv,\curl\tau+\gradoh s) = (\vv,\curl\tau)-(\div\vv,s), \quad
\tau\in \0\Sigma_h, s\in S_h.
$$
In view of the discrete Hodge decomposition \eqref{eq:HodgeA},
$\PV\vv\in\0\VV_h$ is well defined for any $\vv\in\Hodiv$.
It may be characterized as well by the equations
\begin{gather}
\label{orthog1}
(\vv-{\PV{\vv}}, \curl \tau) =0, \qquad \tau \in \0 \Sigma_h,
\\
\label{orthog2}
(\div[\vv-{\PV{\vv}}], s) =0, \qquad s \in S_h.
\end{gather}
Similar projectors have been used elsewhere, e.g.,
\cite[eq.~(2.6)]{CopelGopalOh10}. The properties of $\PV$ are summarized in
the following theorem.

\begin{thm}  \label{thm:proj}
For $\vv\in\Hodiv$ and $U\in \0H^1$,
\begin{equation}\label{eq:commute}
 \div\PV\vv = \PS\div\vv, \quad P_{\0V_h}\curl U = \curl P_{\0\Sigma_h} U.
\end{equation}
Moreover, the following estimates hold
 \begin{gather}
%      \label{eq:PV}
%      \|{\PV{\vv}}\| + h^2 \|\div {\PV{\vv}}\|^2 
%      \le C \big( \|\vv\|^2 + h^2 \|\div \vv\|^2\big),
%      \\
      \label{eq:PV-Lp}
      \|\vv-{\PV{\vv}}\|_{L^p}
      \le C p h^l \|\vv\|_{W^l_p}, \quad 1 \le l \le r,\ 2\le p < \infty,
%       \\
%      \label{eq:PV-Linf}
%      \|\vv-{\PV{\vv}}\|_{L^{\infty}}
%      \le C |\ln h| h^l \|\vv\|_{W^l_{\infty}}, \quad 1 \le l \le r,
         \\
      \label{eq:PV-div}
      \|\div(\vv-{\PV{\vv}})\| 
      \le C h^{l} \|\div \vv\|_{l}, \quad 0 \le l \le r,
  \end{gather}
whenever the norm on the right hand side is finite.
\end{thm}
\begin{proof}
The first commutativity property in \eqref{eq:commute} is immediate from
\eqref{orthog2}, and the divergence estimate \eqref{eq:PV-div} follows
immediately.  For the second commutativity property, we note that
$\curl P_{\0\Sigma_h} U\in\0V_h$ and that, if we set $\vv=\curl U$
and replace $P_{\0V_h}\vv$ by $\curl P_{\0\Sigma_h} U$, then the defining
equations \eqref{orthog1}, \eqref{orthog2} are satisfied.
%Next, apply the discrete and exact decompositions~\eqref{eq:1}
%and~\eqref{eq:Hodge0} to $\vv\in\Hodiv$ and $\PV\vv$, resp., to get $\vv =
%\curl \rho + \grad \phi$ and ${\PV{\vv}} = \curl \rho_h + \gradoh \phi_h.$
%Then
%\begin{align*}
%    \|{\PV{\vv}}\|^2 &= ({\PV{\vv}}-\vv, {\PV{\vv}}) + (\vv, {\PV{\vv}})
%    = ({\PV{\vv}}-\vv, \gradoh \phi_h) + (\vv, {\PV{\vv}})
%    \\
%    & \le (\|{\PV{\vv}}\| + \|\vv\|)
%    \|\gradoh \phi_h \| + \|\vv\| \|{\PV{\vv}}\|.
%  \end{align*}
%Observe that the pair $(\gradoh \phi_h, \phi_h)$ is the mixed finite element
%approximation to $(\grad \phi, \phi)$, where $\phi$ solves the Neumann problem
%given by \eqref{bvps}; cf.~\eqref{eq:bvpsd}. Hence, by standard results,
%\begin{equation*}
%  \begin{aligned}
%    \|\gradoh \phi_h\|
%    & \le \|\gradoh \phi_h - \grad \phi\| + \|\grad \phi\|
%    \le C h \|\grad \phi\|_1 + \|\grad \phi\|
%    \\
%    & \le C h \|\div \vv\| + \|\vv\|.
%  \end{aligned}
%\end{equation*}
%The estimate \eqref{eq:PV} follows by combining these results and using the
%arithmetic-geometric mean inequality.

To prove the $L^p$ estimate \eqref{eq:PV-Lp}, we follow the proof of
corresponding results for mixed finite element approximation of second order
elliptic problems given in~\cite{duran}.  First, we introduce the canonical
interpolant $\PiV:H^1(\Omega;\R^2)\to V_h$ into the Raviart--Thomas space,
defined through the degrees of freedom
\begin{equation}\label{eq:dofs}
\vv\mapsto \int_e \vv\cdot \nn\,w\,ds, \quad w\in\P_{r-1}(e), \qquad
\vv\mapsto \int_T\vv\cdot \ww\,dx, \quad \ww\in \P_{r-2}(T),
\end{equation}
where $e$ ranges over the edges of the mesh and $T$ over the triangles.
Then
\begin{equation}\label{eq:4}
\|\vv - \PiV \vv\|_{L^p} \le C h^l \|\vv\|_{W^l_p}, 
\quad 1 \le l \le r, \ 1 \le p \le \infty,
\end{equation}
and, since  $\div \PiV \vv = \PS \div \vv$, 
\begin{equation}
\label{eq:5}
\|\div(\vv - \PiV \vv)\|_{L^p} \le C h^l \|\div \vv\|_{W^l_p}, 
\quad
0 \le l \le r, \ 1 \le p \le \infty.
\end{equation}
Writing $\vv - \PV{\vv} = (\vv - \PiV \vv) + (\PiV \vv -  \PV{\vv})$,
it thus remains to bound the second term.
From \eqref{eq:commute}, $\div({\PV{\vv}} - \PiV \vv) =0$,
 so
$\PV{\vv} - \PiV \vv = \curl \rho_h$ for some $\rho_h \in \0 \Sigma_h$.
Applying the decomposition~\eqref{eq:1}, we have $\vv - \PiV\vv =
\curl \rho + \grad \psi$ for some $\rho \in \0H^1$ and $\psi \in \hat
H^1$.  From \eqref{orthog1}, 
\begin{equation*}
(\curl \rho_h, \curl \tau) = (\curl \rho, \curl \tau), \quad
\tau \in \0 \Sigma_h.
\end{equation*}
Thus, $\rho_h=\PSigo\rho$ and so satisfies the
bound
$\|\curl \rho_h\|_{L^p} \le C \|\curl \rho\|_{L^p}$ given above in
\eqref{w1p}.

Since
\begin{equation*}
(\curl \rho, \curl \tau) = (\vv - \PiV\vv,\curl \tau)
= (\rot(\vv - \PiV\vv), \tau), \quad \tau \in \0 H^1,
\end{equation*}
$\rho \in \0 H^1$ satisfies $-\Lap \rho = \rot(\vv - \PiV\vv)$.
Using the elliptic regularity result of \cite[Corollary 1]{fromm}, we have
for $1 < p < \infty$ that
\begin{equation*}
\|\rho\|_{W^1_p}\le C_p \|\rot(\vv-\PiV\vv)\|_{W^{-1,p}}
\le C_p \|\vv-\PiV\vv\|_{L^p}.
\end{equation*}
Following the proof of that result, the dependence of the
constant $C_p$ on $p$ arises from the use of the Marcinkiewicz interpolation
theorem for interpolating between a weak $L^1$ and an $L^2$ estimate.
Using the explicit bound on the constant in this theorem found in
\cite[Theorem VIII.9.2]{dibenedetto}, it follows directly that
$C_p \le C p$, where $C$ is a constant independent of $p$.
We remark that this regularity result requires the assumed
convexity of $\Omega$, and does not hold for all $1 <
p < \infty$ if $\Omega$ is only Lipschitz (c.f.~\cite{jk}).  Estimate
\eqref{eq:PV-Lp} follows by combining these results and applying
\eqref{eq:4}.
%
%To obtain \eqref{eq:PV-Linf}, we use an inverse estimate
%to write for $1 \le p < \infty$,
%\begin{multline*}
%\|\PV{\vv} - \PiV \vv\|_{L^{\infty}} 
%\le C h^{-2/p} \|\PV{\vv} - \PiV \vv\|_{L^{p}}
%\le C p h^{-2/p} \|\vv-\PiV\vv\|_{L^p} 
%\\
%\le C p h^{l-2/p} \|\vv\|_{W_p^l}
%\le C p h^{l-2/p} \|\vv\|_{W_{\infty}^l}.
%\end{multline*}
%The result follows by choosing $p = |\ln h|$ and applying \eqref{eq:4}
%and the triangle inequality.
\end{proof}

Theorem~\ref{thm:proj1} below gives one more property of $\PV$, inspired
by an idea in~\cite{scholz}.  To prove it we need a simple lemma.
\begin{lem}
 Let $\rho$ be a piecewise polynomial function with respect to
some triangulation which is nonzero only on triangles meeting $\partial\Omega$.
Then for any $1\le q\le 2$,
$$
\|\rho\|_{L^q}\le C h^{1/q-1/2}\|\rho\|_{L^2},
$$
where the constant $C$ depends only on the polynomial degree and the shape
regularity of the triangulation.
\end{lem}
\begin{proof}
By scaling and equivalence of norms on a finite dimensional space, we have
$$
\|\rho\|_{L^q(T)}\le C h^{2/q-1}\|\rho\|_{L^2(T)}, \quad \rho\in\P_r(T),
$$
where the constant $C$ depends only on the polynomial degree $r$ and the
shape constant for the triangle $T$.
Now, let $\Th^\partial$ denote the set of triangles meeting $\partial\Omega$.
Then
\begin{equation*}
\|\rho\|_{L^q(\Omega)}^q = \sum_{T\in\Th^\partial}\|\rho\|_{L^q(T)}^q
\le Ch^{2-q}\sum_{T\in\Th^\partial}\|\rho\|_{L^2(T)}^q.
\end{equation*}
Applying H\"older's inequality we have
$$
\sum_{T\in\Th^\partial}\|\rho\|_{L^2(T)}^q
\le 
(\#\Th^\partial)^{(2-q)/2}\bigl(\sum_{T\in\Th^\partial}
\|\rho\|_{L^2(T)}^2\bigr)^{q/2},
$$
and $\#\Th^\partial\le Ch^{-1}$ by the assumption of shape regularity.  
Combining these results gives the lemma.
\end{proof}

\begin{thm}\label{thm:proj1}
Let $2 \le p \le \infty$.  Then
  \begin{equation}
   (\vv - {\PV{\vv}}, \curl \tau)
       \le C h^{-1/2-1/p} \|\vv - {\PV{\vv}}\|_{L^p} \|\tau\|,
      \quad \tau \in \Sigma_h, \quad \vv\in \Hodiv \cap L^p.
\end{equation}
\end{thm}
\begin{proof}
Define $\0 \tau \in \0 \Sigma_h$ by taking the Lagrange degrees of freedom to
be the same as those for $\tau$, except setting equal to zero those associated
to vertices or edges in $\partial\Omega$.  Then $\|\0\tau\|\le C\|\tau\|$ and
$\tau-\0\tau$ is nonzero only on triangles meeting $\partial\Omega$.  By
\eqref{orthog1},
\begin{equation*}
(\vv - {\PV{\vv}}, \curl \tau) = (\vv - {\PV{\vv}}, \curl [\tau - \0 \tau]).
\end{equation*}
Let $q=p/(p-1)$, so $1\le q\le 2$.  
Applying H\"older's inequality, the lemma, and an inverse inequality, we obtain
\begin{align*}
(\vv - {\PV{\vv}} , \curl (\tau - \0 \tau ))
& \le \|\vv - {\PV{\vv}}\|_{L^{p}} 
\|\curl(\tau - \0 \tau)\|_{L^q}
\\
& \le C \|\vv - {\PV{\vv}}\|_{L^p} 
%h^{2/s' -1}
%h^{(1/s'-1/s)}   %% s' > 2, so 1/s' < 1/2
h^{1/2 -1/p} 
\|\curl(\tau - \0 \tau)\|_{L^{2}}
\\
& \le C \|\vv - {\PV{\vv}}\|_{L^p} h^{-1/2-1/p} 
\|\tau - \0 \tau\|_{L^{2}},
% \\
% & \le C \|\vv - {\PV{\vv}}\|_{L^{s}(\Omega)} h^{-1/s} 
% \|\tau - \0 \tau\|_{L^2(\partial \Omega)}
% \\
% & \le C h^{-1/2-1/s} \|\vv - {\PV{\vv}}\|_{L^{s}(\Omega)} \|\tau - \0 \tau\|.
\end{align*}
from which the result follows.
\end{proof}

\subsection{Error estimates by an energy argument}
\label{basic}

Using the projection operators defined in the last subsection and
the stability result of the preceding section, we
now obtain a basic error estimate (which is not, however, of optimal order).

\begin{thm}
 \label{energy-rates}
% For $\Sigma_h$ the space of continuous piecewise polynomials of degree
% $\le k+1$ and $\VV_h$ the Raviart-Thomas space of order $k$, we get
% for 
Let $r\ge1$ denote the polynomial degree.  There exists a constant 
$C$ independent of the mesh size $h$ and of $p\in[2,\infty)$, for which
\begin{multline*}
\|\sigma - \sigma_h\| + h \|\sigma-\sigma_h\|_1
+ \|\uu - \uu_h\|_{\Hdiv}
\\
\le 
C
\begin{cases}
h^{l-1/2-1/p} \left(p \|\uu\|_{W_{p}^l} +
       \|\uu\|_{l+1/2-1/p}\right), \ 2\le l\le r, & \text{if $r\ge 2$,}
\\
h^{1/2-1/p} \left( p \|\uu\|_{W_{p}^1} + h^{1/2+1/p} \|\uu\|_{2} \right),
 & \text{if $r= 1$},
\end{cases}
\end{multline*}
whenever the norms on the right hand side are finite.
\end{thm}
\begin{proof}
We divide the errors into the projection and the remainder:
$$
\sigma - \sigma_h = (\sigma - \PSig\sigma)+(\PSig\sigma-\sigma_h),
\quad
\uu - \uu_h = (\uu - \PV\uu)+(\PV\uu-\uu_h).
$$
Since,
$$
\|\sigma- \PSig\sigma\|+h\|\sigma- \PSig\sigma\|_1 \le C h^t\|\sigma\|_t 
\le Ch^t\|\uu\|_{t+1}, \quad 1\le t\le r+1,
$$
and, by Theorem~\ref{thm:proj},
$$
\|\uu-\PV\uu\|_{\Hdiv} \le C h^t\|\uu\|_{t+1}, \quad 1\le t\le r,
$$
the projection error satisfies the necessary bounds 
(without the $p\|\uu\|_{W^l_p}$
term on the right-hand side).

Therefore, setting
$ \rho = \sigma_h - \PSig\sigma$ and $\ww = \uu_h - \PV\uu$,
it suffices to show that for $2\le p < \infty$,
 \begin{equation}\label{eq:diffbd}
  \|\rho\|
  + \|\ww\|_{\Hdiv} \\
    \le C
  \bigl(
      \| \sigma - \PSig \sigma \|
     +
      h \|\sigma-{\PSig{\sigma}}\|_1
      +
      h^{-1/2-1/p} \|\uu - {\PV{\uu}}\|_{L^p}
    \bigr).
 \end{equation}
Indeed, both cases of the theorem follow from~\eqref{eq:diffbd}, 
Theorem~\ref{thm:proj}, and 
the inverse inequality
$Ch\|\rho\|_1\le\|\rho\|$.
By the stability result of
 Theorem~\ref{well-pd}, there exists $(\tau,\vv) \in \Sigma_h
\times \0V_h$ satisfying
  \begin{gather*}
  B(  \rho,\ww ;  \tau,\vv ) \ge c
    \left( 
      \| \rho \|_{\Sigma_h}^2 + \| \ww \|_{\Hdiv}^2
    \right),\quad   \| \tau \|_{\Sigma_h} + \| \vv \|_{\Hdiv} 
     \le C 
    \left(  
      \| \rho\|_{\Sigma_h}
      +
      \| \ww\|_{\Hdiv}
    \right).
  \end{gather*}
By Galerkin orthogonality,
  \begin{align*}
B(\rho,\ww ; \tau,\vv )
& = B( \sigma - \PSig\sigma, \uu - \PV\uu ; \tau,\vv ) \\
&= ( \sigma -     \PSig\sigma, \tau) - (\uu-\PV\uu, \curl \tau) 
   + (\curl (\sigma - \PSig
    \sigma) , \vv) ,
 \end{align*}
  where we used the definition of $B$ and \eqref{eq:commute} in the last step.
Applying  the Cauchy-Schwarz inequality,
Theorem~\ref{thm:proj1},  and \eqref{eq:curlterm},  we then obtain
\begin{multline*}
B( \rho,\ww ;  \tau,\vv ) \le C
    \big(
      \| \sigma - \PSig \sigma \|^2 
      +
      h^2 \|\curl (\sigma-{\PSig{\sigma}} )\|^2
      +
      h^{2(-1/2-1/p)} \|\uu - {\PV{\uu}}\|_{L^p}^2
    \big)^{1/2}
 \\
\x   \left(
      \| \tau \|^2  +  \| \vv \|_{\Hdiv}^2 
    \right)^{1/2}.
\end{multline*}
Together, these imply \eqref{eq:diffbd} and so complete the proof 
of the theorem.
\end{proof}

Choosing $p=|\ln h|$ in the theorem gives a limiting estimate.
\begin{cor}
The following estimates hold whenever the right hand side norm is finite:
% For $\Sigma_h$ the space of continuous piecewise polynomials of degree
% $\le k+1$ and $\VV_h$ the Raviart-Thomas space of order $k$, we get
% for 
% % % For polynomial degree $r\ge1$
% % %    there exists a constant $C$ independent ????
\begin{multline*}
\|\sigma - \sigma_h\| + h \|\sigma-\sigma_h\|_1
+ \|\uu - \uu_h\|_{\Hdiv}
\\
\le 
C
\begin{cases}
 h^{l-1/2} \left(|\ln h| \|\uu\|_{W_{\infty}^l} +
       \|\uu\|_{l+1/2}\right),\ 2 \le l \le r, & \text{if $r\ge 2$,}
\\
h^{1/2} \left( |\ln h| \|\uu\|_{W_{\infty}^1} + h^{1/2} \|\uu\|_{2} \right),
& \text{if $r= 1$}.
\end{cases}
\end{multline*}
\end{cor}

For smooth solutions, choosing the maximum value of $l=r$ 
in the corollary gives
suboptimal approximation of $\sigma$ by order $h^{3/2}$, and suboptimal
approximation of $\uu$ and $\div \uu$ by order $h^{1/2}$ (ignoring
logarithms).  In the next section, we show how to improve the $L^2$ error
estimate for $\uu$ to optimal order.  The other estimates are essentially
sharp, as demonstrated by the numerical experiments already presented.

\subsection{Improved estimates for $\uu-\uu_h$}
\label{improved}
Using duality, we can prove the following estimate
for $\uu-\uu_h$ in $L^2$, which is of optimal order (modulo logarithms
for $r=1$).
\begin{thm}
\label{improved-est}
These estimates hold whenever the right hand side norm is finite:
$$
\|\uu-\uu_h\| \le C
\begin{cases}
  h^l \|\uu\|_l, \ 2 \le l \le r, &\text{if $r\ge 2$},
\\
h\left(|\ln h|^{5/2} \|\uu\|_{W^1_{\infty}} 
+ \|\uu\|_{2}\right), &\text{if $r=1$}.
\end{cases}
$$
\end{thm}
\begin{proof}
Define $\phi\in\Sigma$, $\ww\in\Hodiv$ by
\begin{equation*}
B(\tau,\vv;\phi,\ww)  = (\vv, \uu-\uu_h), 
 \quad \tau \in \Sigma,\ \vv \in \Hodiv.
\end{equation*}
Thus $\ww$ solves the Poisson equation $-\Delta\ww=\uu-\uu_h$ in $\Omega$
with homogeneous Dirichlet boundary conditions, and $\phi=-\rot\ww$.
Under our assumption that $\Omega$ is a convex polygon, 
we know that $\ww\in H^2$, $\phi\in H^1$, and
$\|\phi\|_1 + \|\ww\|_2 \le C \|\uu-\uu_h\|$.

Choosing $\tau=\sigma-\sigma_h$ and $\vv=\uu-\uu_h$ and then using Galerkin
orthogonality, we obtain
\begin{align*}
\|\uu-\uu_h\|^2 &=B(\sigma-\sigma_h,\uu-\uu_h;\phi,\ww)
= B(\sigma-\sigma_h,\uu-\uu_h;\phi-\PSig\phi,\ww-\PV\ww).
\end{align*}
The right hand side is the sum of following four terms:
\begin{align*}
 T_1 &= (\sigma-\sigma_h,\phi-\PSig\phi), 
& T_2 &=-(\uu-\uu_h,\curl[\phi-\PSig\phi]),
\\
 T_3 & = (\curl[\sigma-\sigma_h],\ww-\PV\ww), 
& T_4 &=(\div[\uu-\uu_h],\div[\ww-\PV\ww]).
\end{align*}
We have replaced $\langle\cdot,\cdot\rangle$ by the $L^2$-inner products
because $\phi \in H^1$ and $\sigma=\rot \uu$ is in $H^1$ whenever the
right hand side norm in the theorem is finite.
For $T_1$, we use the
Cauchy--Schwarz inequality, the bound $\|\phi-\PSig\phi\|\le Ch \|\phi\|_1\le
Ch\|\uu-\uu_h\|$ for the elliptic projection, and the estimate 
of Theorem~\ref{energy-rates} with $p=2$ to obtain 
\begin{gather*}
|T_1|\le
C
\begin{cases}
h^l \|\uu\|_l \|\uu-\uu_h\|, \ 2\le l \le r, &\text{if $r\ge 2$,}
\\
h(\|\uu\|_1 + h\|\uu\|_2)\|\uu-\uu_h\|, &\text{if $r =1$.}
\end{cases}
\end{gather*}
Similar considerations give the same bound for $T_4$.

To bound $T_2$, we split it as $(\PV\uu -\uu,\curl[\phi-\PSig\phi])$ and
$T_2'=(\uu_h-\PV\uu,\curl[\phi-\PSig\phi])$. The first term is clearly bounded
by $Ch^l\|\uu\|_l\|\uu-\uu_h\|$, while, for the second, we use
\eqref{eq:curlterm} to find that
$$
|T_2'|\le C h \|\curl(\phi-\PSig\phi)\|\|\div(\uu_h-\PV\uu)\|.
$$
Bounding $\div(\uu_h-\PV\uu)$ via Theorem~\ref{energy-rates} and 
\eqref{eq:PV-div}, we get
\begin{align*}
|T_2'| &\le Ch^l \|\uu\|_l \|\uu-\uu_h\|, \quad 2 \le l \le r,
\\
|T_2'| &\le Ch(\|\uu\|_1 + h\|\uu\|_2)\|\uu-\uu_h\|, \quad r =1.
\end{align*}
Finally, we bound $T_3$.  If $r\ge 2$, then we simply use the Cauchy--Schwarz
inequality, the bound
\begin{equation}\label{eq:wb}
\|\ww-\PV\ww\|\le Ch^2\|\ww\|_2\le Ch^2 \|\uu-\uu_h\|,
\end{equation}
and the $p=2$ case of Theorem~\ref{energy-rates} to obtain
\begin{equation*}
|T_3|\le Ch^l \|\uu\|_l \|\uu-\uu_h\|, \quad 2\le l\le r.
\end{equation*}
If $r=1$, then \eqref{eq:wb} does not hold.  Instead we
split $T_3$ as  $(\curl[\sigma-\PSig\sigma],\ww-\PV\ww)+
 (\curl[\PSig\sigma-\sigma_h],\ww-\PV\ww)$.  Since
$\|\ww-\PV\ww\|\le Ch\|\ww\|_1\le Ch\|\uu-\uu_h\|$,
the first term is bounded by $C h \|\sigma\|_1 \|\uu-\uu_h\|
\le C h \|\uu\|_2 \|\uu-\uu_h\|$. For the second, 
we apply Theorem~\ref{thm:proj1} and \eqref{eq:PV-Lp} to obtain
\begin{multline*}
|(\curl[\PSig\sigma-\sigma_h],\ww-\PV\ww)|
\le C h^{-1/2-1/p}\|\ww-\PV\ww\|_{L^p}\|\PSig\sigma-\sigma_h\|
\\
\le C h^{1/2-1/p} p \|\ww\|_{W_p^1}\|\PSig\sigma-\sigma_h\|, \quad 2\le p<\infty.
%\le C h^{1/2-1/p} p^2 \|\uu - \uu_h\|_{W_p^{-1}}\|\PSig\sigma-\sigma_h\|
\end{multline*}
By the Sobolev inequality, $\|\ww\|_{W_p^1}\le K_p\|\ww\|_{W^2_q}$, where
$q=2p/(2+p)<2$.  Moreover, from \cite{talenti} and a simple extension
argument the constant $K_p\le Cp^{1/2}$.  Since $\|\ww\|_{W^2_q}\le C \|\ww\|_2$ with
$C$ depending only on the area of the domain, we obtain
$$
|(\curl[\PSig\sigma-\sigma_h],\ww-\PV\ww)|\le C h^{1/2-1/p} p^{3/2}\|\PSig\sigma-\sigma_h\|\|\uu-\uu_h\|, \quad 2\le p<\infty.
$$
By \eqref{eq:PS-err} and Theorem~\ref{energy-rates} with $r=1$,
$$
\|\PSig\sigma-\sigma_h\|\le \|\sigma-\PSig\sigma\|+\|\sigma-\sigma_h\| \le C (h^{1/2-1/p}p\|\uu\|_{W^1_p} + h \|\uu\|_2).
$$
Thus we obtain 
$$
|T_3|\le C\left(h^{1-2/p} p^{5/2}  \|\uu\|_{W^1_p}+ h^{3/2-1/p}p^{3/2}\|\uu\|_{2}\right)
\|\uu-\uu_h\|, \quad 2 \le p < \infty,
$$
and, by choosing $p= |\ln h|$ and noting that $h^{1/2}|\ln h|^{3/2}$ is bounded,
$$
|T_3|\le Ch  \left( |\ln h|^{5/2} \|\uu\|_{W^1_{\infty}} + \|\uu\|_{2}\right)
\|\uu-\uu_h\|.
$$
The theorem follows easily from these estimates.
 \end{proof}

\section{The Ciarlet-Raviart Mixed Method for the Biharmonic}
\label{sec:biharmonic}
In this section, we show that the above analysis immediately gives
estimates for the Ciarlet--Raviart mixed method for the biharmonic,
including some new estimates which improve on those available in
the literature.

Given $g \in H^{-2}(\Omega) = (\0 H^2(\Omega))'$, the
standard weak formulation of the Dirichlet problem for the 
biharmonic seeks $U \in \0 H^2$ such that
\begin{equation*}
(\Lap U, \Lap V) = (g,V), \quad V \in \0 H^2.
\end{equation*}
Letting $\sigma:=-\Delta U\in L^2$, we have $\Delta\sigma=-g$.
Assuming that $g \in H^{-1}(\Omega)$, as we henceforth shall,
%$\sigma\in\tilde\Sigma:= \{\,\tau\in L^2\,|\,\Delta\tau\in H^{-1}\,\}$.
%(The space $\tilde\Sigma$ can be shown to coincide with the space $\Sigma$
%defined earlier if the domain $\Omega$ has a $C^{1,1}$ boundary
%\cite[Prop.~7.2]{d-s-s-2}, but we do
%not assume such smoothness.)
for $\Omega$ a convex polygon, we have that $U \in H^3(\Omega)$,
$\sigma \in H^1(\Omega)$ and
\begin{equation*}
\|U\|_3 + \|\sigma\|_1 \le C \|g\|_{-1}.
\end{equation*}
Hence $(\sigma,U)\in H^1\times \0 H^1$ satisfy
\begin{gather*}
(\sigma, \tau) - (\curl U, \curl \tau) =0, \quad \tau \in H^1,
\\
(\curl \sigma, \curl V) = (g, V), \quad V \in \0 H^1.
\end{gather*}
We note that a mixed formulation in these variables, but with spaces that are
less regular, can also be given for this problem (c.f.~\cite{b-g-m}),
but we shall not pursue this approach here.
%Now, if $\tau\in\tilde\Sigma$, then, since $U\in\0H^2$,
%$$
%\<\Delta\tau,U\> = (\tau,\Delta U) = -(\tau,\sigma).
%$$
%Thus we find that $\sigma\in\tilde\Sigma$, $U\in \0H^1$ satisfy
%\begin{equation}\label{crweak}
%\begin{gathered}
% (\sigma, \tau) + \<\Delta \tau,U\> =0, \quad \tau \in \tilde\Sigma,
%\\
%\<\Delta\sigma, V\>  = -(g, V), \quad V \in \0 H^1.
%\end{gathered}
%\end{equation}
%This mixed weak formulation of the biharmonic problem was proposed
%in \cite{b-g-m}.}

The Ciarlet--Raviart mixed method \cite{ciarlet-raviart} for the approximation
of the Dirichlet problem for the biharmonic equation using Lagrange elements
of degree $r$, seeks $\sigma_h \in \Sigma_h$, $U_h \in \0 \Sigma_h$ such that
\begin{gather*}
(\sigma_h, \tau) - (\curl U_h, \curl \tau) =0, \quad \tau \in \Sigma_h,
\\
(\curl \sigma_h, \curl V) = (g, V), \quad V \in \0 \Sigma_h.
\end{gather*}
%Since $\<\Delta\tau, V\> = -(\curl\tau,\curl V)$ whenever $\tau\in H^1$ and
%$V\in\0H^1$, this is a conforming Galerkin discretization of the problem
%\eqref{crweak}.
This discretization has been analyzed in many papers under the assumption
that $\Omega$ is a convex polygon. It is proved in
\cite{falk-osborn} and \cite{b-o-p} that for $r \ge 2$,
\begin{equation*}
\|U-U_h\|_1 \le C h^{r} \|U\|_{r+1}, \qquad
\|\sigma - \sigma_h\| \le C h^{r-1} \|U\|_{r+1}.
\end{equation*}
The former estimate is optimal, while the estimate for 
$\|\sigma-\sigma_h\|$ is two orders suboptimal.
The case $r=1$ was analyzed in \cite{scholz}, where it was proven that
\begin{equation*}
\|U-U_h\|_1 \le C h^{3/4} |\ln h|^{3/2} \|U\|_{4}, \qquad
\|\sigma - \sigma_h\| \le C h^{1/2} |\ln h| \|U\|_{4}.
\end{equation*}
These estimates are suboptimal by $1/4$ and $3/2$ orders respectively (modulo
logarithms) and require $H^4$ regularity of $U$.  (As noted in \cite{scholz},
the same technique could be applied for $r\ge 2$ to obtain a $3/2$ suboptimal
estimate on $\|\sigma-\sigma_h\|$.)  Below we improve the estimate
on $\|U-U_h\|_1$ for $r=1$ to an optimal order estimate (modulo logarithms), 
with decreased assumptions on the regularity of the solution $U$.

We now show how to obtain all of these results from the analysis
of the previous section, with only minor modifications.
Let $\uu=\curl U$.  Then
$$
 B(\sigma,\uu;\tau,\curl V) = (g,V),
 \quad (\tau,V)\in H^1\x\0H^1.
$$
Similarly, with $\uu_h=\curl U_h$,
\begin{equation*}
 B(\sigma_h,\uu_h;\tau,\curl V) = (g,V), 
\quad (\tau,V)\in\Sigma_h\x\0\Sigma_h.
\end{equation*}
As above, set $\rho=\sigma_h-\PSig\sigma\in\Sigma_h$, 
$\ww=\uu_h-\PV\uu\in\0V_h$.  Note that
$\ww=\curl U_h -\curl\PSigo U\in\curl\0\Sigma_h$.  Subtracting the above equations and
writing $\vv$ for $\curl V$, we have
$$
B(\rho,\ww;\tau,\vv)= B(\sigma-\PSig\sigma,\uu-\PV\uu;\tau,\vv), 
\quad (\tau,\vv)\in \Sigma_h\x\curl\0\Sigma_h.
$$
Since the stability result of Theorem~\ref{well-pd} holds over the space
$\Sigma_h\x\curl\0\Sigma_h$, as stated in the last sentence of
the theorem, we can argue exactly as in proof of
Theorem~\ref{energy-rates} and conclude that the estimates proved in that
theorem for the Hodge Laplacian hold as well in this context with
one improvement.  To estimate the term $\|\uu-\PV\uu\|_{L^p}$ in 
\eqref{eq:diffbd}, instead of using \eqref{eq:PV-Lp}, we note that
$\|\uu-\PV\uu\|_{L^p}=\|\curl(U-\PSigo U)\|_{L^p}$
and invoke \eqref{w1p}.  In this way we avoid a factor of $p$.
The improved estimates of Theorem~\ref{improved-est}
also translate to this problem, with essentially
the same proof and a similar improvement.
The dual problem is, of course, now taken to be: Find
$\phi \in \Sigma, \ww \in \curl \0H^1$ such that
\begin{equation*}
B(\tau, \vv; \phi,\ww) = (\vv, \uu-\uu_h), \qquad \tau \in \Sigma, \, \vv
\in \curl \0H^1.
\end{equation*}
Thus $\ww = \curl W$, where $W$ solves the biharmonic problem $\Lap^2 W =
\rot(\uu-\uu_h) \in H^{-1}$ with Dirichlet boundary conditions, and $\phi =
\Lap W$.  The relevant regularity result, valid on a convex domain, is
\begin{equation*}
\|\ww\|_2 + \|\phi\|_1 \le C\|W\|_3 \le C \|\rot(\uu-\uu_h)\|_{-1} 
\le C \|\uu-\uu_h\|.
\end{equation*}
The remainder of the proof goes through as before, with the simplification
that now the terms $T_4$ and $T_2^{\prime}$ are zero, and the term
$\|\ww-\PV\ww\|_{L^p}$ can be bounded without introducing a factor of $p$ as
just described.  The suppressed factors of $p$ lead to fewer logarithms in the
final result.  Stating this result in terms of the original variable $U$
instead of $\uu=\curl U$, we have the following theorem.

\begin{thm}
 \label{energy-rates-b}
Let $U$ solve the Dirichlet problem for the biharmonic equation, 
$\sigma=-\Delta U$, and let $U_h\in\0\Sigma_h$, $\sigma_h\in\Sigma_h$ denote the
discrete solution obtained by the Ciarlet--Raviart mixed method with Lagrange
elements of degree $r\ge 1$.  If $r\ge2$ and $2 \le l \le r$, then the
following estimates, requiring differing amounts of regularity, hold
whenever the norms on the right hand side are finite:
\begin{align*}
\|\sigma - \sigma_h\| + h \|\sigma-\sigma_h\|_1
&\le 
C\left\{\begin{array}{l}
       h^{l-1} \|U\|_{l+1}
       \\
       h^{l-1/2} \left(\|U\|_{W_{\infty}^{l+1}} + \|U\|_{l+3/2}\right)
\end{array}\right.,
\\
\|U - U_h\|_1 &\le C h^l \|U\|_{l+1}.
\\
\intertext{If $r=1$, the estimates are:}
\|\sigma - \sigma_h\| + h \|\sigma-\sigma_h\|_1
&\le 
C\left\{\begin{array}{l}
       \left(\|U\|_{2} + h \|U\|_{3}\right),
       \\
       h^{1/2} \left( \|U\|_{W_{\infty}^2} + h^{1/2} \|U\|_{3} \right)
\end{array}\right.,
\\
\|U - U_h\|_1 &\le Ch (|\ln h|^{1/2}\|U\|_{W_{\infty}^{2}} +\|U\|_3).
\end{align*}
\end{thm}

\section{Stationary Stokes equations}
\label{sec:stokes}

Another application in which the vector Laplacian with Dirichlet boundary
conditions arises is the stationary Stokes equations, in which the vector
field represents the velocity, subject to no-slip conditions on the boundary.
A standard weak formulation (with viscosity equal to one) seeks $\uu \in \0
H^1(\Omega, \R^2)$ and $p \in \hat L^2$ such that
\begin{gather*}
(\grad \uu, \grad \vv) - (p, \div \vv) = (\ff, \vv), 
\quad \vv \in \0 H^1(\Omega, \R^2),
\\
(\div \uu, q) =0, \quad q \in L^2.
\end{gather*}
Mixed methods, such as we have discussed, have been used to approximate this
problem, based on the vorticity-velocity-pressure formulation. For example,
using the spaces defined in Section~\ref{sec:analysis}, the following weak
formulation is discussed in \cite{d-s-s-2}.  Find $\sigma \in \Sigma$, $\uu
\in \Hodiv$, $p \in \hat L^2$ such that
\begin{gather*}
(\sigma, \tau) - \<\curl \tau,\uu\> =0, \quad \tau \in \Sigma,
\\
\<\curl \sigma, \vv\> -(p, \div \vv) = (\ff,\vv), \quad
\vv \in \Hodiv.
\\
(\div \uu, q) = 0, \quad q \in L^2.
\end{gather*}
This formulation is obtained just as for the vector Laplacian, by
writing
\begin{equation*}
(\grad \uu, \grad \vv) = (\rot \uu, \rot \vv) + (\div \uu, \div \vv)
\end{equation*}
and introducing the variable $\sigma = \rot \uu$.
When $\ff \in L^2(\Omega;\R^2)$ and $\Omega$ is a convex polygon, $\uu \in
H^2(\Omega;\R^2)$, $p \in \hat H^1(\Omega)$, and $\sigma = \rot \uu \in
H^1(\Omega)$.  Assuming this extra regularity, and setting $\uu = \curl U$,
and $\vv = \curl V$, $(\sigma, U) \in H^1 \times \0 H^1$ satisfy the stream
function-vorticity equations:
%(first posed in these spaces in \cite{b-g-m}):
\begin{gather*}
(\sigma, \tau) - (\curl U, \curl \tau) =0, \quad \tau \in H^1
\\
(\curl \sigma, \curl V) = (\ff,\curl V), \quad V \in \0 H^1.
\end{gather*}
Taking $g = \rot \ff$, this formulation coincides with the mixed formulation of
the biharmonic problem discussed in the previous section.  

We consider here the finite element approximation which
seeks $\sigma_h \in \Sigma_h$, $\uu_h \in \0{\VV}_h$, $p_h \in \hat S_h$
such that
\begin{gather*}
(\sigma_h, \tau) - (\uu_h, \curl \tau) =0, \quad \tau \in \Sigma_h,
\\
(\curl \sigma_h, \vv) -(p_h, \div \vv)  = (\ff,\vv), \quad \vv \in \0{\VV}_h.
\\
(\div \uu, q) = 0, \quad q \in\hat S_h .
\end{gather*}
where the spaces $\Sigma_h$, $\0{\VV}_h$, and $\hat S_h$ are defined as above.
The existence and uniqueness of the solution is easily
established by standard methods.  When $\ff=0$, we get by choosing $\tau =
\sigma_h$, $\vv = \uu_h$, $q=p_h$ and adding the equations that $\sigma_h=0$
and $\div \uu_h =0$.  Hence $\uu_h = \curl U_h$, $U_h \in \0 \Sigma_h$, and
choosing $\tau = U_h$, we see that $\curl U_h =0$.  Since $\div \0 \VV_h = \hat
S_h$, we also get $p_h=0$.

Error estimates for $\|\uu-\uu_h\|$ and $\|\sigma- \sigma_h\|$ are easily
obtained by reducing the problem to its stream function-vorticity form
and using the estimates obtained in the previous section. Letting
$\uu_h = \curl U_h$, and choosing $\vv = \curl V$, $V \in \0\Sigma_h$,
we see that $(\sigma_h, U_h)$ is the unique solution
of the Ciarlet-Raviart formulation of the biharmonic with $g = \rot \ff$.
Hence, the estimates for $\sigma-\sigma_h$ in Theorem~\ref{energy-rates-b}
remain unchanged, except that we can replace $\|U\|_s$ by $\|\uu\|_{s-1}$.
In particular, we have the following theorem.
\begin{thm}
 \label{energy-rates-s}
Let $(\uu,p)$ solve the Dirichlet problem for the Stokes equation, $\sigma=\rot
\uu$, and let $\uu_h\in\0\VV_h$, $\sigma_h\in\Sigma_h$, and $p_h \in \hat S_h$
denote the discrete solution obtained by the vorticity-velocity-pressure mixed
method with $r\ge 1$ the polynomial degree.  If $r\ge2$ and $2 \le l \le
r$, then the following estimates, requiring differing amounts of regularity,
hold whenever the norms on the right hand side are finite:
\begin{align*}
\|\sigma - \sigma_h\| + h \|\sigma-\sigma_h\|_1
&\le 
C\left\{\begin{array}{l}
       h^{l-1} \|\uu\|_{l},
       \\
       h^{l-1/2} \left(\|\uu\|_{W_{\infty}^{l}} + \|\uu\|_{l+1/2}\right),
\end{array}\right.
\\
\|\uu - \uu_h\| &\le C h^l \|\uu\|_{l}.
\\
\intertext{If $r=1$, the estimates are:}
\|\sigma - \sigma_h\| + h \|\sigma-\sigma_h\|_1
&\le 
C\left\{\begin{array}{l}
       \|\uu\|_{1} + h \|\uu\|_{2},
       \\
       h^{1/2} \left( \|\uu\|_{W_{\infty}^1} + h^{1/2} \|\uu\|_{2} \right),
\end{array}\right.
\\
\|\uu - \uu_h\| &\le Ch (|\ln h|^{1/2}\|\uu\|_{W_{\infty}^{1}} +\|\uu\|_2).
\end{align*}
\end{thm}

The only item remaining is to derive error bounds for the approximation
of the pressure. We obtain the following result, which gives error
bounds that are suboptimal by $O(h^{1/2})$.
\begin{thm}
\label{pressure-est}
If $r\ge2$ and $2 \le l \le r$, then
\begin{equation*} 
\|p-p_h\| \le C \begin{cases}
h^{l-1} (\|\uu\|_l + \|p\|_{l-1}), 
\\
h^{l-1/2} \left(\|\uu\|_{W_{\infty}^l} +
       \|\uu\|_{l+1/2} + \|p\|_{l-1/2} \right).
\end{cases}
\end{equation*}
If $r=1$, the estimates are
\begin{equation*} 
\|p-p_h\| \le C \begin{cases}
\|\uu\|_1 + h \|\uu\|_2 + \|p\|,
\\
h^{1/2} \left(\|\uu\|_{W_{\infty}^1} +
       h^{1/2}\|\uu\|_{2} + \|p\|_{1/2} \right).
\end{cases}
\end{equation*}
\end{thm}
\begin{proof}
From the variational formulation, we get the error equation
\begin{equation*}
(p_h - \PS p, \div \vv_h) = (p - \PS p, \div \vv_h)
+ (\curl[\sigma_h-\sigma], \vv_h), \quad \vv_h \in \0\VV_h.
\end{equation*}
We choose 
$\vv \in \0H^1(\Omega;\R^2)$ such that $\div \vv = p_h - \PS p$ and
$\|\vv \|_1 \le C \|p_h - \PS p\|$, and take
$\vv_h = \PiV\vv$.  We have that $\div \vv = \div \PiV\vv$ and
$\|\PiV\vv\|_{\Hdiv} \le C \|\vv\|_1 \le C \|p_h - \PS p\|$, so
\begin{align*}
\|p_h - \PS p\|^2 &= (p_h - \PS p, \div \PiV\vv)
= (p - \PS p, \div \PiV\vv)
+ (\curl[\sigma_h-\sigma], \PiV\vv),
\\
&= (p - \PS p, p_h - \PS p) + (\curl[\sigma_h-\sigma], \PiV\vv - \vv)
+ (\sigma_h-\sigma, \rot \vv)
\\
&\le C(\|p - \PS p\| +  h \|\curl(\sigma_h-\sigma)\| + \|\sigma_h-\sigma\|)
\|p_h - \PS p\|.
\end{align*}
It easily follows that
\begin{equation*}
\|p - p_h\| \le C(\|p - \PS p\| + \|\sigma_h-\sigma\| 
+ h \|\curl(\sigma_h-\sigma)\|).
\end{equation*}
The theorem follows directly by applying Theorem~\ref{energy-rates-s}
and standard estimates for the error in the $L^2$ projection.
\end{proof}

A number of papers have been devoted to finite element approximation schemes
of either the vorticity-velocity-pressure or stream-function-vorticity
formulation of the Stokes problem.  
% See, for example, \cite{d-s-s-1},
% \cite{d-s-s-3}, \cite{s-s}, and \cite{s-s2} and the references, therein.
In particular, the lowest order ($r=1$) case of the method analyzed here was
discussed in \cite{d-s-s-1}, (in which additional references can also be
found).  In the case of the magnetic boundary conditions,
$\sigma=\uu\cdot\nn=0$, the authors established stability and first-order
convergence, which is optimal, for all variables.  But for the no-slip
boundary conditions $\uu=0$, with which we are concerned and which arise much
more commonly in Stokes flow, they observe in numerical experiments
stability problems and reduced rates of convergence which are in agreement
with the theory presented above.

We close with a simple numerical example in the case $r=2$ that demonstrates
that the suboptimal convergence orders obtained above are sharp even for very
smooth solutions.
%The second shows that even this convergence is degraded when the solution has
%a corner singularity.  For both the 
Our discretization of the vorticity-velocity-pressure mixed formulation of the
Stokes problem then approximates the velocity $\uu$ by the second lowest order
Raviart--Thomas elements, the vorticity $\sigma$ by continuous piecewise
quadratic functions, and the pressure $p$ by discontinuous
piecewise linear functions.  We take $\Omega$ to be the unit square and compute
$\ff$ corresponding to the polynomial solution velocity field $\uu = (-2x^2
(x-1)^2 y (2y - 1) (y - 1), 2y^2 (y-1)^2 x (2x - 1)(x - 1))$, and pressure $p
= (x-1/2)^5 +(y-1/2)^5$.  The computations, summarized in Table~\ref{tb:t4},
indeed confirm the convergence rates established above,
i.e., $\uu_h$ converges with optimal order $2$ to $\uu$ in $L^2$, while the
approximations to $\sigma$ and $\curl \sigma$ are both suboptimal by $3/2$
order and the approximation to the pressure $p$ is suboptimal by $1/2$ order.
%Table 4
\begin{table}[!ht]
\footnotesize
\caption{$L^2$ errors and convergence rates for the mixed finite element
approximation of the Stokes problem for the vector Laplacian 
with boundary conditions $\uu \cdot \nn =0$, $\uu \cdot \ssb =0$
on the unit square.}
\label{tb:t4}
\begin{center}
\begin{tabular}{>{\centering}p{.85in}c|>{\centering}p{.85in}c|>{\centering}p{.85in}c|>{\centering}p{.85in}c}%{cc|cc|cc|cc}
$\|\uu-\uu_h\|$ & rate & $\|p-p_h\|$ & rate 
& $\|\sigma-\sigma_h\|$ & rate & $\|\curl(\sigma-\sigma_h)\|$ & rate\\[.05in]
\hline\vrule height.2in width0in
%$3.43e-03$ &  & $1.46e-02$ & $   $ & $1.48e-02$ & $   $ & $2.24e-01$ &     \\
%$1.20e-03$ & $1.5$ & $5.86e-03$ & $1.3$ & $6.88e-03$ & $1.1$ & $1.88e-01$ & $0.3$ \\
3.26e-04 & 1.9 & 2.34e-03 & 1.3 & 2.70e-03 & 1.3 & 1.67e-01 & 0.2 \\
8.35e-05 & 2.0 & 8.05e-04 & 1.5 & 9.70e-04 & 1.5 & 1.24e-01 & 0.4 \\
2.10e-05 & 2.0 & 2.74e-04 & 1.6 & 3.47e-04 & 1.5 & 8.96e-02 & 0.5 \\
5.27e-06 & 2.0 & 9.39e-05 & 1.6 & 1.24e-04 & 1.5 & 6.42e-02 & 0.5 \\
\end{tabular}
\end{center}
\end{table}

\bibliographystyle{amsplain}
\bibliography{vectorlaplacian}
\end{document}